\documentclass[11pt]{article}
\usepackage{lmodern}
\usepackage[a4paper]{geometry}
\usepackage{amsmath,amssymb,amsthm}
\usepackage{tikz,mathrsfs,hyperref}
\usepackage[mathcal]{euscript} 
\usepackage{mathtools}
\usepackage{etoolbox}
\usepackage{stmaryrd}

\renewcommand{\c}[1]{\mathcal{#1}}
\renewcommand{\bf}[1]{\textup{\textbf{#1}}}
\newcommand{\bb}[1]{\mathbb{#1}}
\renewcommand{\rm}[1]{\textrm{#1}}
\renewcommand{\sf}[1]{\textup{\textsf{#1}}}

\newcommand{\inn}[1]{\langle #1\rangle}
\newcommand{\deq}{\mathrel{\mathop:}=}
\renewcommand{\models}{\mathrel{{|}\kern-0.47ex{=}\kern-1.725ex{=}}}
\newcommand{\nmodels}{\mathrel{{|}\kern-0.47ex{=}\kern-1.725ex{\neq}}}
\renewcommand{\le}{\leqslant}
\renewcommand{\ge}{\geqslant}
\renewcommand{\b}{\breve{\ }}
\newcommand{\m}{\wedge}
\renewcommand{\j}{\vee}

\newcommand{\pw}{\raise0.545ex\hbox{{\scalebox{1.0925}[1.0925]{\(\wp\)}}}}

\newcommand{\id}{\mathrm{id}}

\DeclareMathOperator{\At}{\bf{At}}
\DeclareMathOperator{\at}{\textup{At}}

\DeclareMathOperator{\pru}{\mathrm{Pr}_\mathrm{U}}
\DeclareMathOperator{\prl}{\mathrm{Pr}_\mathrm{L}}

\usetikzlibrary{arrows,shapes,positioning}
\usetikzlibrary{decorations.markings}
\tikzstyle arrowstyle=[scale=1]
\tikzstyle directed=[postaction={decorate,decoration={markings,
    mark=at position 0.535 with {\arrow[arrowstyle]{to}}}}]

\makeatletter
\newcommand{\leqnomode}{\tagsleft@true\let\veqno\@@leqno}
\newcommand{\reqnomode}{\tagsleft@false\let\veqno\@@eqno}
\makeatother

\makeatletter 
\patchcmd\@thm
  {\let\thm@indent\indent}{\let\thm@indent\noindent}%
  {}{}
\patchcmd{\@begintheorem}{\scshape}{}{}{}
\patchcmd{\@opargbegintheorem}{\scshape}{}{}{}
\makeatother
\expandafter\patchcmd\csname\string\proof\endcsname
  {\normalparindent}{0pt }{}{}
\expandafter\patchcmd\csname\string\theorem\endcsname
  {\normalparindent}{0pt }{}{}

\newtheorem{theorem}{Theorem}
\newtheorem{lemma}[theorem]{Lemma}
\newtheorem{corollary}[theorem]{Corollary}
\newtheorem{proposition}[theorem]{Proposition}

\newtheorem{definition}[theorem]{Definition}

\newtheorem{problem}{Problem}

\reqnomode 
 
\begin{document}

\begin{center}
\Large{\textbf{Finite Relation Algebras}}\\

James M. Koussas\\

 La Trobe University\\
 
 jameskoussas@gmail.com

\end{center}

\section{Abstract}

We will show that almost all nonassociative relation algebras are symmetric and integral (in the sense that the fraction of both labelled and unlabelled structures that are symmetric and integral tends to \(1\)), and using a Fra{\"i}ss{\'e} limit, we will establish that the classes of all atom structures of nonassociative relation algebras and relation algebras both have \(0\)--\(1\) laws. As a consequence, we obtain improved asymptotic formulas for the numbers of these structures and broaden some known probabilistic results on relation algebras.

\section{Introduction}

The calculus of relations is a branch of logic that was developed in the nineteenth century, largely due to the work of Augustus De Morgan, Charles Peirce, and Ernst Schr{\"o}der; see \cite {morg}, \cite{peirce}, and \cite{schroder}, for example. By the beginning of twentieth century, the calculus of relations was considered to be an important branch of logic. Indeed, in \cite{russell}, Bertrand Russell stated that ``the subject of symbolic logic is formed by three parts: the calculus of
propositions, the calculus of classes, and the calculus of relations.'' The calculus of relations even played a role in the birth of model theory. Indeed, Leopold L{\"o}wenheim stated and proved the earliest known version of the L{\"o}wenheim-Skolem Theorem as a result on the calculus of relations; for more details, see \cite{bad}. Interest in the field mostly faded until the publication of \cite{og}, where Alfred Tarski defined an abstract algebraic counterpart to the calculus of relations, namely relation algebras. Tarski and many of his students took an interest in these algebras, which~lead to relation algebras becoming a fairly popular area of research that is still active.

The idea of defining the probability of a property holding in a class of finite structures as a limit is due to Rudolf Carnap (see \cite{carnap}), but similar ideas appeared earlier. In~\cite{mtprob}~and~\cite{mtnumb}, Ronald Fagin looked at the probability of certain sentences holding in the classes of all relational structures of a given type as well as asymptotic formulas for the numbers of these structures. The study of \(0\)--\(1\) laws was initiated by  Glebski\u{\i}, Kogan, Liogon'ki\u{\i}, and Talanov in the nineteen seventies; see \cite{og1} and \cite{og2}. The study of conditions that guarantee the existence of \(0\)--\(1\) laws for first-order and monadic second-order properties were studied extensively by Kevin Compton in \cite{comp1} and \cite{comp2}. In the context of relation algebras, there has been surprisingly little research published on probabilistic results of this nature.  In~\cite{fira},  Roger Maddux finds an asymptotic formula for the number of nonassociative relation algebras in which the identity is an atom and the number of integral relation algebras, and shows that almost all of these algebras are rigid and satisfy any finite set of equations that hold in all representable relation algebras. In the present article, we show that almost all finite nonassociative relation algebras are symmetric and have \(e\) as an atom, and establish a \(0\)--\(1\) law for the class of all atom structures of finite  nonassociative relation algebras. Combining these results with the results of Maddux, we obtain a simple asymptotic formula for both the number of nonassociative relation algebras and the number of relation algebras, and show that almost all finite nonassociative relation algebras are integral relation algebras.

\section{Preliminaries}

We begin by giving a formal definition of labelled and unlabelled probabilities of properties. We will mostly follow the approach taken in \cite{2kinds}. Note that by \(\bb{N}\) we mean \(\{1,2,\dots\}\).

\begin{definition}[Probabilities]
Let \(\c{K}\) be a class of finite structures of a finite signature \(F\) that is closed under isomorphism and has no upper bound on the size of its members.  For all \(n \in \bb{N}\), let \(\c{U}_n\) be a set with precisely one representative from each isomorphism class of \(n\)-element structures from \(\c{K}\) and let \(\c{L}_n\) be the set of all structures in \(\c{K}\) with universe~\(\{1,\dots,n\}\). Let \(P\) be some property of \(F\)-structures that is invariant under isomorphisms (for example, a~first-order property). Let \(s \colon \bb{N} \to \bb{N}\) be the increasing sequence of values of \(n\) with \(\c{U}_n \neq \varnothing\). If the limit
\[
\lim_{n \to \infty} \frac{|\{ \bf{A} \in \c{U}_{s(n)} \mid \bf{A} \models P\}|}{|\c{U}_{s(n)}|}
\]
exists, we call it the \emph{unlabelled probability of \(P\)} and denote it by \(\pru(P,\c{K})\). If the limit
\[
\lim_{n \to \infty} \frac{|\{ \bf{A} \in \c{L}_{s(n)} \mid \bf{A} \models P\}|}{|\c{L}_{s(n)}|}
\]
exists, we call it the \emph{labelled probability of \(P\)} and denote it by \(\prl(P,\c{K})\). If \(\pru(P,\c{K}) = 1\), we say that \emph{almost all structures in \(\c{K}\) satisfy \(P\)}.
\end{definition}

We will mostly work with classes where there are elements of every possible cardinality, so the sequence \(s\) will be the identity sequence. 

The result we will need from Freese \cite{2kinds} is stated below; this result is stated for algebras, but the proof also works for relational structures.  Similar results appear in earlier articles, such as Fagin \cite{mtnumb}.

\begin{proposition}\label{rigid}
Let \(\c{K}\) be a class of similar finite structures of a finite signature \(F\) that is closed under isomorphism and has no upper bound on the size of its members, let \(R\) be the property of being rigid (i.e., having a trivial automorphism group), let \(P\) be a property \(F\)-structures that is invariant under isomorphism, and assume that we have \(\prl(R,\c{K}) = 1\). If one of \(\pru(P,\c{K})\) and \(\prl(P,\c{K})\) exists, then both quantities exist and are equal.
\end{proposition}

Now we are in the position to recall the definiton of an almost sure theory.

\begin{definition}[Almost sure theory]
Let \(\c{K}\) be a class of finite structures of a finite signature \(F\) that is closed under isomorphism and has no upper bound on the size of its members. We call the set of all first-order sentences \(\sigma\) with \(\prl(\sigma,\c{K}) = 1\) the \emph{almost sure theory of \(\c{K}\)}.
\end{definition}

Next, we will give a brief introduction to relation algebras. For a more extensive introduction, we refer the reader to Hirsch and Hodkinson~\cite{rabg}, Maddux~\cite{madduxRA}, and Givant~\cite{intRA}. We begin by defining relation algebras. To match the notation used in lattice theory, we will use \(\j\) and \(\m\) rather than \(+\) and \(\cdot\). This allows us to use a more group theoretic notation by using \(\cdot\) and \(e\)  rather than \(;\) and \(1'\). We will assume that unary operations are applied first and use multiplicative notation~for~\(\cdot\). Therefore \(x\b z \m y'\) means \(((x \b) \cdot z) \m (y ')\), for example. We call \(d \deq e'\) the \emph{diversity element}. We use \(\approx\) for logical equality, as in universal algebra. To avoid reusing symbols, we wil use \(\curlyvee\) for logical disjunction, \(\curlywedge\) for logical conjunction, and \(\neg\) for logical negation. Where it is possible, we will usually drop superscripts on operations.

\begin{definition}[Relation algebras] 
An algebra \(\bf{A} = \inn{A;\j,\m,\cdot,{'},{\b}, 0,1,e}\) is called a \emph{nonassociative relation algebra} iff \(\inn{A;\j,\m,{'},0,1}\) is a Boolean algebra, \(e\) is an identity element~for~\(\cdot\), and the \emph{triangle laws} hold, i.e., we have
\[
xy \m z = 0 \iff x\b z \m y = 0 \iff zy\b \m x = 0,
\]
for all \(x,y,z \in A\). The class of all nonassociative relation algebras will be denoted by \(\sf{NA}\). An algebra \(\bf{A} \in \sf{NA}\) is called a \emph{relation algebra} iff \(\cdot\) is associative.  The class of all relation algebras will be denoted by \(\sf{RA}\). An algebra \(\bf{A} \in \sf{NA}\) is said to be \emph{symmetric} iff \(\bf{A} \models x\b \approx x\).
\end{definition}

We extend ideas from Boolean algebra to these algebras in the obvious way. For example, an atom of a nonassociative relation algebra is an atom of its Boolean algebra reduct.

Some basic properties of these algebras are summarised in the following result.

\begin{proposition}\label{basicRA}
Let \(\bf{A} \in \sf{NA}\).
\begin{enumerate}
\item
\(\bf{A} \models x(y \j z) \approx xy \j xz\) and \(\bf{A} \models (x \j y)z \approx xz \j yz\).
\item
\(\bf{A} \models (x \j y)\b \approx x\b \j y\b\).
\item
\(\bf{A} \models 0\b = 0\), \(\bf{A} \models 1\b = 1\), \(\bf{A} \models e\b = e\), and \(\bf{A}\models d\b = d\).
\item
\(\bf{A} \models x\b\b \approx x\).
\item
If \(a\) is an atom, then \(a\b\) is an atom.
\item
If \(\bf{A} \models e \approx 0\), then \(\bf{A}\) is trivial.
\end{enumerate}
\end{proposition}

Based on Proposition \ref{basicRA}, the operations of a complete atomic (and, in particular, a finite) nonassociative relation algebra are completely determined by their values on its atoms. Since a finite \(\bf{A}\in \sf{NA}\) has \(\log_2(|A|)\) atoms, this means these algebras are determined by a small subset of its elements. This motivates the following definitions. When \(e\) is an atom, it is sometimes convenient to include it in the signature rather than a unary relation.
\begin{definition}[Atom structure]
Let \(\bf{A}\) be a complete atomic nonassociative relation algebra. We call \(\At(\bf{A}) \deq \inn{\at(\bf{A}); f_\bf{A}, I_{\bf{A}}, T_{\bf{A}}}\) the \emph{atom structure} of \(\bf{A}\), where \(\at(\bf{A})\) denotes the set of all atoms of \(\bf{A}\), \(f_\bf{A}\) is defined by \(x \mapsto x\b\), \(I_{\bf{A}} \deq \{ a \in \at(\bf{A}) \mid a \le e\}\), and \(T_\bf{A} \deq \{ (a,b,c) \in \at(\bf{A})^3 \mid ab \ge c\}\). If \(e\) is an atom, we put \(\At_e(\bf{A}) \deq \inn{\at(\bf{A}); f_\bf{A}, e,T_\bf{A}}\).
\end{definition}

It turns out that these structures can be axiomatised. 

\begin{definition}[\(\sf{FAS}\), \(\sf{FSIAS}\), and \(\sf{FSIAS}_e\)] 
Let \(\sf{FAS}\) denote the class of all finite structures of the signature \(\{f,T,I\}\) (where \(f\) is a unary operation symbol, \(I\) is a unary relation symbol, and \(T\) is a ternary relation symbol) that satisfy
\begin{enumerate}
\item[\textup{(P)}]
for all \(a,b,c \in U\), we have \((f(a),c,b) , (c,f(b),a) \in T\) whenever \((a,b,c) \in T\),
\item[\textup{(I)}]
for all  \(a,b \in U\), we have \(a = b\) if and only if \(i \in I\) with \((a,i,b) \in T\).
\end{enumerate}
We call a \(\{f,T,I\}\)-structure \(\bf{U}\) \emph{integral} iff we have \(|I| = 1\), and \emph{symmetric} iff \(\bf{U} \models f(x) \approx x\). The class of all symmetric and integral members of \(\sf{FAS}\) will be denoted by \(\sf{FSIAS}\). Now, let \(\sf{FSIAS}_e\) be the class of all finite structures of the signature \(\{f,e,T\}\)  (where \(f\) is a unary operation symbol, \(e\) is a nullary operation symbol, i.e., a constant, and \(T\) is a ternary relation symbol) that satisfy \(f(x) \approx x\) and
\begin{enumerate}
\item[\textup{(IP)}]
for all \(a,b,c \in U\), we have \((f(a),c,b) , (c,f(b),a) \in T\) whenever \((a,b,c) \in T\),
\item[\textup{(II)}]
for all  \(a,b \in U\), we have \(a = b\) if and only if there is some \((a,e,b) \in T\).
\end{enumerate}
\end{definition}

The above definition abuses language slightly; a non-trivial \(\bf{A} \in \sf{NA}\) is called \emph{integral} iff  \(xy = 0\) implies that \(x=0\) or \(y = 0\), which is equivalent to \(e\) being an atom when \(\bf{A} \in \sf{RA}\), but not in general. We refer to Maddux \cite{madduxRA} and Maddux \cite{varcontRA} for further details.

\begin{proposition}\label{atcm}
\(\sf{FAS}\) is precisely the class of all atom structures of finite members~of~\(\sf{NA}\), \(\sf{FSIAS}\) is precisely the class of all atom structures of finite, integral, and symmetric members of \(\sf{NA}\). There are bijective correspondences between the sets of isomorphism classes from:
\begin{enumerate}
\item
the class of finite members of \(\sf{NA}\) and \(\sf{FAS}\);
\item
the class of finite, integral, and symmetric members of \(\sf{NA}\) and \(\sf{FISAS}\);
\item
\(\sf{FSIAS}\) and \(\sf{FSIAS}_e\).
\end{enumerate}
\end{proposition}

Next, we introduce the notion of a cycle (from Maddux \cite{fira}) which can be used to define and describe these structures.

\begin{definition}[Cycles]
Let \(\bf{U}\) be a \(\{f,T,I\}\)-structure and let \(a,b,c \in U\). We call \((a,b,c)\), \((f(a),c,b)\), \((b,f(c),f(a))\), \((f(b),f(a),f(c))\), \((f(c),a,f(b))\), and \((c,f(b),a)\) the \emph{Peircean transforms} of \((a,b,c)\). The set of all of these triples is called a \emph{cycle}, and is denoted by  \([a,b,c]\).  We call \((a,b,c)\) an \emph{identity triple} iff  \(I \cap \{a,b,c\} \neq \varnothing\), and a \emph{diversity triple} otherwise. We call \([a,b,c]\) an \emph{identity cycle} iff it contains an identity triple, and a \emph{diversity cycle} otherwise. We call \((a,b,c)\) \emph{consistent} iff \((a,b,c) \in T\), and \emph{forbidden} otherwise. We call \([a,b,c]\) \emph{consistent} iff  \([a,b,c] \subseteq T\), and \emph{forbidden} iff \([a,b,c] \cap T = \varnothing\). We call \(a\) an \emph{identity atom} iff \(a \in I\) and a \emph{diversity atom} otherwise. We extend these ideas to \(\{f,e,T\}\) in the obvious way. 
\end{definition}

The following result from \cite{fira} illustrates the connection between cycles and the axioms for \(\sf{FAS}\) and \(\sf{FSIAS}_e\).

\begin{proposition}\label{asprop}
\begin{enumerate}
\item
Let \(\bf{U}\) be an \(\{f,T,I\}\)-structure. 
\begin{enumerate}
\item
The following are equivalent:
\begin{enumerate}
\item
\(\bf{U}\) satisfies \textup{(P)};
\item
for all \(a,b,c \in U\), the cycle \([a,b,c]\) is either consistent or forbidden.
\end{enumerate}
\item
The following are equivalent:
\begin{enumerate}
\item
\(\bf{U}\) satisfies \textup{(I)};
\item
for all \(a, b \in U\), we have \(a = b\) if and only if \([a,i,b]\) is consistent, for some \(i \in I\).
\end{enumerate}
\item
If \(\bf{U}\) is integral, then the following are equivalent:
\begin{enumerate}
\item
\(\bf{U}\) satisfies \textup{(I)};
\item
\(f(e) = e\) and \(\{[a,e,a] \mid a \in U\}\) is the set of consistent identity cycles, where \(e\) is the unique element of \(I\).
\end{enumerate}
\end{enumerate}
\item
Let \(\bf{U}\) be a \(\{f,e,T\}\)-structure.
\begin{enumerate}
\item
The following are equivalent:
\begin{enumerate}
\item
\(\bf{U}\) satisfies \textup{(IP)};
\item
for all \(a,b,c \in U\), the cycle \([a,b,c]\) is either consistent or forbidden.
\end{enumerate}
\item
The following are equivalent:
\begin{enumerate}
\item
\(\bf{U}\) satisfies \textup{(II)};
\item
\(f(e) = e\) and \(\{[a,e,a] \mid a \in U\}\) is the set of consistent identity cycles.
\end{enumerate}
\end{enumerate}
\end{enumerate}
\end{proposition}

Based on Proposition~\ref{atcm} and Proposition~\ref{asprop}, once given a finite set \(U\), some \(e \in U\), and an involution \(f \colon U \to U\) with \(f(e) = e\), each \(\bf{U} \in \sf{FAS}\) that is an expansion of \(\inn{U;f,\{e\}}\) is completely determined by which cycles are consistent or forbidden. Using this observation, it is possible count the number of atom-structures of a given (finite) size. Indeed, in \cite{fira}, Maddux obtains asymptotic formulas using this method. The results we will need from \cite{fira} are summarised in the following results.

\begin{proposition}\label{cyclecount}
Let \(U\) be an \(n\)-element set, for some \(n \in \mathbb{N}\), let \(e \in U\), let \(f\) be an involution of \(U\) with \(f(e) = e\), and let \(s \deq |\{a \in U \mid f(a) = a\}|\).
\begin{enumerate}
\item
There are \(s-1\) diversity cycles with \(1\) triple; ones of the form \([a,a,a]\).
\item
There are \((n-s)/2\) diversity cycles with \(2\) triples; ones of the form \([a,a,f(a)]\), where \(f(a) \neq a\).
\item
There are \((s-1)(n-2)\) diversity cycles with \(3\) triples; ones of the form \([a,b,b]\), where \(f(a)=a\) and \(a \neq b\).
\item
There are \((n-1)((n-1)^2-3s+2)/6 + (s-1)/2\) diversity cycles with \(6\) triples.
\item
There are \(Q(n,s) \deq (n-1)((n-1)^2+3s-1)/6\) diversity cycles in total.
\item
There are \(P(n,s) \deq (s-1)! ((n-s)/2)! 2^{(n-s)/2}\) automorphisms of \(\inn{U;f,\{e\}}\).
\end{enumerate}
\end{proposition}

Before we state the next result, we will pause to recall the definition of representability.

\begin{definition}[\sf{RRA}]
Let \(E\) be an equivalence relation over a set \(D\). We call the structure \(\inn{\pw(E); \cup, \cap, {\mid}, {^c}, {^{-1}}, \varnothing, E, \id_D}\) the \emph{proper relation algebra} on \(E\), where  \(\pw(E)\) be the powerset (i.e., set of all subsets) of \(E\), \(\mid\) is relational composition, \({^{c}}\) is set complement relative to \(E\), \({^{-1}}\) is relational inverse, and \(\id_D\) is the identity relation on \(D\). Thus, for each \(R,S \subseteq E\), 
\begin{align*}
R \mid S &= \{ (x,z) \in D^2 \mid (x,y) \in R, (y,z) \in S, \rm{ for some } y \in D\}, \\
R^{-1} &= \{ (y,x) \in D^2 \mid (x,y) \in R\}, \\
\id_D & = \{ (x,y) \in D^2 \mid x = y \}.
\end{align*}
A relation algebra is said to be \emph{representable} iff it embeds into a proper relation algebra. The class of representable relation algebras will be denoted by \(\sf{RRA}\).
\end{definition}

The problem of determining whether every relation algebra is representable was the main focus of research into relation-type algebras for many years, until it was solved in the negative by Lyndon in \cite{lyndon}. Therefore one might be interested in the asymptotics of the fraction of representable relation algebras. Maddux took the first steps in this direction in \cite{fira}.

\begin{proposition}\label{firasumm}
\begin{enumerate}
\item
Almost all integral labelled structures in \(\sf{FAS}\) are rigid.
\item
If \(E\) is the conjunction of a finite set of equations that hold in all members of \(\sf{RRA}\), then \(E\) holds in almost all finite elements of \(\sf{NA}\) in which \(e\) is an atom. In particular, almost all nonassociative relation algebras in which \(e\) is an atom are relation algebras.
\end{enumerate}
\end{proposition}

\begin{proposition}\label{asymp}
For all \(n,s \in \bb{N}\) with \(s \le n\) and \(n-s\) even, let \(F(n,s)\) be the number of isomorphism classes of \(n\)-atom integral relation algebras with \(s\) atoms satisfying \(x \b = x\). Then \(2^{Q(n,s)}/P(n,s)\) is an asymptotic formula for \(F(n,s)\), in the sense that, for all \(\varepsilon > 0\), there exists \(N\in \bb{N}\) such that if \(n,s \in \bb{N}\) with \(n > N\), \(s \le n\), and \(n-s\) even, then
\[
\Bigg| 1 - \frac{F(n,s)P(n,s)}{2^{Q(n,s)}} \Bigg| < \varepsilon.
\]
Further, the same statement holds for nonassociative relation algebras in which \(e\) is an atom.
\end{proposition}

We conclude with a reminder of  Fra{\"i}ss{\'e} limits, which were defined by Roland Fra{\"i}ss{\'e} in \cite{fraisse}. We will mostly follow Hodges \cite{lsmt}. Firstly, we recall some definitions.

\begin{definition}[Age]
Let \(\bf{A}\) be a structure. The \emph{age of \(\bf{A}\)} is the class of all finitely generated structures that embed into \(\bf{A}\).
\end{definition}

\begin{definition}[HP, JEP, and AP]
Let \(\c{K}\) be a class of similar structures. We say that \(\c{K}\) has the \emph{hereditary property} (HP) iff \(\c{K}\) is closed under forming finitely generated structures. We say that \(\c{K}\) has the \emph{joint embedding property} (JEP)  iff, for all \(\bf{A}, \bf{B} \in \c{K}\), there is some \(\bf{C} \in \sc{K}\) that both \(\bf{A}\) and \(\bf{B}\) embed into. We say that \(\c{K}\) has the \emph{amalgamation property} (AP)  iff, for all \(\bf{A},\bf{B},\bf{C} \in \c{K}\) and embeddings \(\mu \colon \bf{A} \to \bf{B}\) and \(\nu \colon \bf{A} \to \bf{C}\), there is some \(\bf{D} \in \c{K}\) and embeddings \(\mu' \colon \bf{B} \to \bf{D}\) and \(\nu' \colon \bf{C} \to \bf{D}\) such that \(\mu' \circ \mu = \nu' \circ \nu\).
\end{definition}

\begin{definition}[Homogeneity]
Let \(\bf{A}\) be a structure. We call \(\bf{A}\)  \emph{ultrahomogenous} iff every isomorphism between finitely generated substructures of \(\bf{A}\) extends to an automorphism of \(\bf{A}\). We call \(\bf{A}\) \emph{weakly homogeneous} iff, for all finitely generated structures \(\bf{B}\) and \(\bf{C}\) of \(\bf{A}\) with \(\bf{B} \le \bf{C}\) and all embeddings \(\mu \colon \bf{B} \to \bf{A}\), there is an embedding \(\nu \colon \bf{C} \to \bf{A}\) extending \(\mu\).
\end{definition}

The following result shows that these definitions coincide

\begin{proposition}\label{homogequiv}
A finite or countable structure is ultrahomogeneous if and only if it is weakly homogeneous.
\end{proposition}

The main result on these structures is the following existence and uniqueness result, known as Fra{\"i}ss{\'e}'s Theorem.

\begin{proposition}[Fra{\"i}ss{\'e}'s Theorem]\label{fraisse}
Let \(S\) be a countable signature and let \(\c{K}\) be a class of at most countable \(S\)-structures, that has the \textup{HP}, \textup{JEP}, and \textup{AP}. Then there is an \(S\)-structure \(\bf{F}\) (called a \emph{Fra{\"i}ss{\'e} limit of \(\c{K}\)}), unique up to isomorphism, such that \(\bf{F}\) is at most countable, \(\c{K}\) is the age of \(\bf{F}\), and \(\bf{F}\) is ultrahomogeneous.
\end{proposition}

To state the last result of this section, we need to recall some definitions.

\begin{definition}[Uniform local finiteness]
Let \(\c{K}\) be a class of similar structures. We say that \(\c{K}\) is \emph{uniformly locally finite} iff there is a function \(f \colon \bb{N}\to \bb{N}\), such that, for all \(\bf{A} \in \c{K}\), each \(n \in \bb{N}\), and every subset \(S\) of \(A\) with \(|S| \le n\), the substructure of \(\bf{A}\) generated by \(S\) has cardinality at most \(f(n)\).
\end{definition}

\begin{proposition}\label{qeomega}
Let \(S\) be a finite signature, let \(\c{K}\) be a uniformly locally finite class of \(S\)-structures with the HP, JEP, and AP, and at most countably many isomorphism types of finitely generated \(S\)-structures, and let \(\bf{F}\) be a Fra{\"i}ss{\'e} limit of \(\c{K}\). Then the first-order theory of \(\bf{F}\) is \(\aleph_0\)-categorical and has quantifier elimination.
\end{proposition}

\section{Main results}\label{mains}

First, we show that almost all finite nonassociative relation algebras are symmetric and have \(e\) as an atom. The proof and the observation that almost all nonassociative relation algebras in which \(e\) as an atom are symmetric were discovered independently by the author, but this result was also conjectured by Roger Maddux in a private communication. 

\begin{theorem}\label{aasi}
Almost all members of \(\sf{FAS}\) are in \(\sf{FSIAS}\).
\end{theorem}

\begin{proof}
Let \(n \ge 5\), let \(U\) be an \(n\)-element set, and let \(1 \le i < n\). Clearly, there are \(\binom{n}{i}\) ways~to~select \(i\) identity atoms. Now, let \(0 \le p \le \lfloor (n-i)/2 \rfloor\). There are at most  \(\binom{n-i}{2}^p\) involutions of \(U\) with \(p\) non-fixed pairs, i.e., with sets of the form \(\{u, f(u)\}\) with \(u \neq f(u)\), since \({\binom{n-i}{2}}^p\) is the number of \(p\) \emph{independent} selections of \(2\)-element sets of diversity atoms. Based on Proposition \ref{asprop}. 1. (b), there are \((2^i - 1)^n\) possible ways of selecting identity cycles, since each element of \(U\) must appear in at least one of the \(i\) cycles of the given form. Lastly, based on Proposition \ref{cyclecount}. 5., there are \(2^{Q(n-i+1,n-i+1-2p)}\) ways to select diversity cycles; the number of choices of diversity cycles in a \((n-i+1)\)-element structure satisfying \(|I|=1\) and a \(n\)-element structures such that \(|I| = i\) and \(|U \setminus I| = n-i\)  is clearly the same. Hence, by Proposition~\ref{basicRA}(5), the fraction of members of \(\sf{FAS}\) with universe \(\{1,\dots,n\}\) that belong to \(\sf{FSIAS}\) is bounded below by
\begin{align*}
&\frac{n2^{Q(n,n)}}{\sum_{i=1}^{n} \sum_{p = 0}^{\lfloor (n-i)/2 \rfloor} \binom{n}{i} \binom{n-i}{2}^p (2^i - 1)^n 2^{Q(n-i+1,n-i+1-2p)}} \\
= & \frac{1}{(\sum_{i=1}^{n} \sum_{p = 0}^{\lfloor (n-i)/2 \rfloor} \binom{n}{i} \binom{n-i}{2}^p (2^i - 1)^n 2^{Q(n-i+1,n-i+1-2p)-Q(n,n)})/n}
\end{align*}

We have
\begin{align*}
& \quad\,\,  \frac{1}{n}\sum_{i=1}^{n} \sum_{p = 0}^{\lfloor (n-i)/2 \rfloor} \binom{n}{i} \binom{n-i}{2}^p (2^i - 1)^n 2^{Q(n-i+1,n-i+1-2p)-Q(n,n)} \\ 
& = \frac{1}{n} \sum_{p=0}^{\lfloor (n-1)/2 \rfloor} n \binom{n-1}{2}^p 2^{Q(n,n-2p)-Q(n,n)} \\
& \quad\,\, +  \frac{1}{n}\sum_{i=2}^{n-1} \sum_{p = 0}^{\lfloor (n-i)/2 \rfloor} \binom{n}{i} \binom{n-i}{2}^p (2^i - 1)^n 2^{Q(n-i+1,n-i+1-2p)-Q(n,n)}.
\end{align*}
Clearly,
\[
\binom{n-1}{2}^p \le \bigg(\frac{n^2}{2}\bigg)^p.
\]
Now,
\begin{align*}
Q(n,n-2p) - Q(n,n) &= \frac{1}{6}(n-1)((n-1)^2 + 3(n-2p) - 1) - \frac{1}{6}(n-1)((n-1)^2 + 3n -1) \\
&= \frac{1}{6}(n-1)((n-1)^2 +3n-6p-1 -((n-1)^2 +3n-1)) \\
&= \frac{1}{6}(n-1)(3n -6p - 3n) \\
&= (1-n)p,
\end{align*}
hence
\begin{align*}
\bigg( \frac{n^2}{2} \bigg)^p 2^{Q(n,n-2p)- Q(n,n)} &= \bigg( \frac{n^2}{2} \bigg)^p 2^{(1-n)p} \\
&= \bigg(\frac{n^2}{2^n}\bigg)^p.
\end{align*}
Since \(n \ge 5\), we have \(0 < n^2/2^n < 1\), so the formula for a geometric sum gives
\begin{align*}
\frac{1}{n} \sum_{p=0}^{\lfloor (n-1)/2 \rfloor} n \binom{n-1}{2}^p 2^{Q(n,n-2p)-Q(n,n)}&\le \sum_{p=0}^{\lfloor (n-1)/2 \rfloor} \bigg(\frac{n^2}{2^n}\bigg)^p \\
&= \frac{1 - (n^2/2^n)^{\lfloor (n-1)/2 \rfloor +1}}{1 - n^2/2^n}.
\end{align*}
Using basic limits, \(n^2/2^n\) and \((n^2/2^n)^{\lfloor (n-1)/2 \rfloor +1}\) tend to \(0\), and so
\[
 \lim_{n \to \infty} \bigg(\frac{1 - (n^2/2^n)^{\lfloor (n-1)/2 \rfloor +1}}{1 - n^2/2^n}\bigg) = 1.
\]

Define \(S(m) \deq Q(m,m)\), for each \(m \in \bb{N}\). We have
\begin{align*}
S(m)  &= \frac{1}{6}(m-1)((m-1)^2+3m-1) \\
&= \frac{1}{6}(m-1)(m^2-2m+1+3m-1)\\
&= \frac{1}{6}(m-1)(m^2+m) \\
&= \frac{1}{6}(m^3-m),
\end{align*}
for each \(m \in \bb{N}\). Using the formula for a difference of cubes, we get
\begin{align*}
S(n-i+1) - S(n) & = \frac{1}{6}((n-i+1)^3 - (n-i+1) - n^3+n)\\
& = \frac{1}{6}((n-i+1-n)((n-i+1)^2 + n(n-i+1) + n^2) + i - 1) \\
& = \frac{1}{6}(-(i-1)(n^2 -2(i-1)n + (i-1)^2 + n^2 - (i-1)n + n^2) + i-1)\\
& = \frac{1}{6}(-3(i-1)n^2 + 3(i-1)^2n - (i-1)^3 + i-1).
\end{align*}
In particular,
\begin{align*}
i = 2 & \implies S(n-i+1) - S(n) = - \frac{1}{2}n^2+\frac{1}{2}n, \\
i = 3 & \implies S(n-i+1) - S(n) = -n^2+2n-1, \\
i = 4 & \implies S(n-i+1) - S(n) = - \frac{3}{2}n^2+ \frac{9}{2}n -4.
\end{align*}
If \(1 \le i < n\) and \(1 \le p \le \lfloor (n-i)/2 \rfloor\), then \(\lfloor (n-i)/2 \rfloor \le n\),  \(\binom{n}{i} \le n^n\), \(\binom{n-i}{2}^p \le (n^2)^n = n^{2n}\), \((2^i-1)^n \le 2^{in}\), and \(Q(n-i+1,n-i+1-2p) \le S(n-i+1)\). Over the interval \([1,\infty)\), \(x \mapsto (x^3-x)/6\) is increasing, hence \(S(n-i+1)-S(n)\) is maximised when \(i\) is minimised. Since \(n \ge 5\), we have \(n-1 \ge 4\) and \(2^n > 1\). Hence, using the formula for a geometric sum, we get
\begin{align*}
& \quad \,\, \frac{1}{n}\sum_{i=2}^{n} \sum_{p = 0}^{\lfloor (n-i)/2 \rfloor} \binom{n}{i} \binom{n-i}{2}^p (2^i - 1)^n 2^{Q(n-i+1,n-i+1-2p)-Q(n,n)} \\
& \le \frac{1}{n} \sum_{i=2}^{n} \sum_{p=0}^{\lfloor (n-i)/2 \rfloor} n^n  n^{2n}  (2^i-1)^n 2^{S(n-i+1)-S(n)} \\
& \le \frac{1}{n} \sum_{i=2}^{n} n  n^{3n}  (2^i-1)^n 2^{S(n-i+1)-S(n)} \\
& = \sum_{i=2}^{n} n^{3n} (2^i-1)^n 2^{S(n-i+1)-S(n)} \\
& = n^{3n} 3^{n} 2^{-n^2/2+n/2} + n^{3n}7^{n} 2^{-n^2 + 2n -1} + \sum_{i=4}^{n} n^{3n} (2^i-1)^n 2^{S(n-i+1)-S(n)} \\
& \le n^{3n} 3^{n} 2^{-n^2/2+n/2} + n^{3n}7^{n} 2^{-n^2 + 2n -1} + \sum_{i=4}^{n} n^{3n} 2^{in} 2^{-3n^2/2+9n/2-4} \\
& \le n^{3n} 3^{n} 2^{-n^2/2+n/2} + n^{3n}7^{n} 2^{-n^2 + 2n -1} + n^{3n}2^{-3n^2/2+9n/2-4} \sum_{i=0}^{n} (2^n)^i \\ 
& = n^{3n} 3^{n} 2^{-n^2/2+n/2} + n^{3n}7^{n} 2^{-n^2 + 2n -1} + n^{3n}2^{-3n^2/2+9n/2-4} \frac{2^{n(n+1)}-1}{2^n-1}.
\end{align*}
We have
\[
 n^{3n} 3^{n} 2^{-n^2/2+n/2} = 2^{3n\log_2(n) + \log_2(3)n-n^2/2+n/2},
\]
which clearly tends to \(0\). Similarly,
\[
n^{3n} 7^n  2^{-n^2+2n+1} = 2^{3n\log_2(n)+\log_2(7)n -n^2+2n-1},
\]
which tends to \(0\). Lastly,
\begin{align*}
n^{3n}2^{-3n^2/2+9n/2-4} \frac{2^{n(n+1)}-1}{2^n-1} &= 2^{3n\log_2(n) - n^2/2+11n/2-4} \frac{2^{-n^2-n}(2^{n^2+n}-1)}{2^n-1} \\
& = 2^{3n\log_2(n)-n^2/2 +11n/2 -4} \frac{1-2^{-n^2-n}}{2^{n}-1}.
\end{align*}
Now, it is clear that \(2^{3n\log_2(n)-n^2/2+9n/2-4}\) tends to \(0\) and \((1-2^{-n^2-n})/(2^n-1)\) tends to \(0\), hence the term above has limit \(0\). Combining these results with basic limits, we get
\[
\lim_{n\to \infty} \Bigg( \frac{1}{n}\sum_{i=1}^{n} \sum_{p = 1}^{\lfloor (n-i)/2 \rfloor} \binom{n}{i} \binom{n-i}{2}^p (2^i - 1)^n 2^{Q(n-i+1,n-i+1-2p)-Q(n,n)}\Bigg) \le 1,
\]
and so
\[
\lim_{n \to \infty} \Bigg(\frac{1}{(\sum_{i=1}^{n} \sum_{p = 1}^{\lfloor (n-i)/2 \rfloor} \binom{n}{i} \binom{n-i}{2}^p (2^i - 1)^n 2^{Q(n-i+1,n-i+1-2p)-Q(n,n)})/n}\Bigg) \ge 1.
\]
By the Squeeze Principle, the fraction of members of \(\sf{FAS}\) with universe \(\{1,\dots,n\}\) in \(\sf{FSIAS}\) tends to \(1\). Combining these results, Proposition~\ref{rigid}, Proposition~\ref{atcm}, and Proposition~\ref{firasumm}(1), we find that almost all members of \(\sf{FAS}\) belong to \(\sf{FSIAS}\), which is what we wanted.
\end{proof}

Combining this with Proposition \ref{atcm} and Proposition \ref{firasumm}. 2., we obtain the following.

\begin{corollary}\label{symint}
Almost all finite nonassociative relation algebras are symmetric integral relation algebras.
\end{corollary}

Using Proposition \ref{asprop} and Proposition \ref{asymp}, we obtain the following. 

\begin{corollary}
\(2^{Q(n,n)}/(n-1)!\) is an asymptotic formula for the number of \(n\)-atom nonassociative relation algebras. The same formula holds if we add the assumption of associativity, symmetry, or \(e\) being an atom.
\end{corollary}

Next, we aim to establish a \(0\)--\(1\) law for \(\sf{FAS}\). Based on Proposition~\ref{aasi}, it will be enough to establish a \(0\)--\(1\) law for \(\sf{FSIAS}\). We essentially follow the method outlined in the introduction of Bell and Burris in \cite{bellburr}. For the completeness proof required for this method, we will make use of a Fra{\"i}ss{\'e} limit. It will be convenient to work with the class \(\sf{FSIAS}_e\) rather than \(\sf{FSIAS}\), then translate the result, as \(\sf{FSIAS}\) does not have the HP. First, we show that a limit exists.

\begin{lemma}
\(\sf{FSIAS}_e\) has a Fra{\"i}ss{\'e} limit.
\end{lemma}

\begin{proof}
By Theorem \ref{fraisse}, it is enough to show that \(\sf{FSIAS}_e\) has the HP, JEP, and AP.

By definition, \(\sf{FSIAS}_e\) is the class of finite members of a universal class. Based on this,  \(\sf{FSIAS}_e\) is closed under forming substructures, so \(\sf{FSIAS}_e\) clearly has the HP.

For the AP, let \(\bf{S},\bf{V},\bf{W} \in \sf{FSIAS}_e\) and let \(\mu \colon \bf{S} \to \bf{V}\) and \(\nu \colon \bf{S} \to \bf{W}\) be embeddings. Without loss of generality, we can assume that \(V \cap W = S\) and \(\mu\) and \(\nu\) are inclusion maps. Thus, we can define \(\bf{U} \deq \inn{U; f^\bf{U},e, T^\bf{U}}\), where  \(U \deq V \cup W\),  \(f^\bf{U}\) is given by
\[
f^\bf{U}(x) = \begin{cases}  f^\bf{V}(x) & \rm{if } x \in V \\ f^\bf{W}(x) & \rm{if } x \in W, \end{cases}
\]
and \(T^\bf{U} \deq T^\bf{V} \cup T^\bf{W}\).  Let \(a,b,c \in U\) and assume that \((a,b,c) \in T^\bf{U}\). By construction, we~have (\(a,b,c \in V\) and \((a,b,c) \in T^\bf{V}\)) or (\(a,b,c \in W\) and \((a,b,c) \in T^\bf{W}\)). In the first case, 

\noindent \((f^\bf{U}(a),c,b), (c,f^\bf{U}(b),a) \in T^\bf{U}\), since \(T^\bf{V} \subseteq T^\bf{U}\), \(f^\bf{U}{\upharpoonright}_V = f^\bf{V}\), and \(\bf{V}\) satisfies (IP). Similarly, we have \((f^\bf{U}(a),c,b),(c,f^\bf{U}(b),a) \in T^\bf{U}\) in the second case, so \(\bf{U}\) satisfies (IP). Let \(a \in U\). If \(a \in V\), then we have \((a,e,a) \in T^\bf{U}\), since \(T^\bf{V} \subseteq T^\bf{U}\) and \(\bf{V}\) satisfies (II). Similarly, \((a,e,a) \in T^\bf{U}\) when \(a \in W\). Since \(U = V \cup W\), it follows that \((a,e,a) \in T^\bf{U}\) in every case. Lastly, let \(a,b \in U\) such that \((a,e,b) \in T^\bf{U}\). By construction, (\(a,b \in V\) and \((a,e,b) \in T^\bf{V}\)) or (\(a,b \in W\) and \((a,e,b) \in T^\bf{W}\)). Since \(\bf{V}\) and \(\bf{W}\) satisfy (II), we have \(a = b\), so (II) holds. Since \(\bf{V}\) and \(\bf{W}\) both satisfy \(f(x) \approx x\), it follows that \(f^\bf{V}\) and \(f^\bf{W}\) are both identity maps. By construction, \(f^\bf{U}\) is an identity map, so \(\bf{U} \models f(x) \approx x\). By definition, \(U = V  \cup W\), hence \(|U| \le |V|+|W|\). Thus, \(U\) is finite. Based the above results, we have \(\bf{U} \in \sf{FSIAS}_e\). Clearly, the inclusion maps \(\imath_V \colon V \to U\) and \(\imath_W \colon W \to U\) are embeddings and \(\imath_V \circ \mu =  \imath_W \circ \nu\). Combining these results, we find that \(\sf{FSIAS}_e\) has the AP, which is what we wanted to show.

Clearly, \(\sf{FSIAS}_e\) contains a trivial structure. This structure embeds into all \(\bf{A} \in \sf{FSIAS}_e\), so the JEP follows from the AP. Thus, \(\sf{FSIAS}_e\) has the HP, JEP and AP, as required. \end{proof}

\begin{definition}[\(\bf{L}_\sf{SI}\), \(T_\sf{SI}\), and \(S_\sf{SI}\)]
Let \(\bf{L}_\sf{SI}\) be a Fra{\"i}ss{\'e} limit of \(\sf{FSIAS}_e\),  let \(T_\sf{SI}\) be the first-order theory of \(\bf{L}_\sf{SI}\), and let \(S_\sf{SI}\) be the almost sure theory of \(\sf{FSIAS}_e\).
\end{definition}

The elements of \(\sf{FSIAS}_e\) are symmetric, so generated substructures contain at most one extra element, namely \(e\). So, by Proposition \ref{qeomega}, we have the following.

\begin{corollary}
\(T_\sf{SI}\) is \(\aleph_0\)-categorical and has quantifier elimination.
\end{corollary}

Next we introduce what Bell and Burris call extention axioms in \cite{bellburr}. The sentences essentially assert that a substructure can be extended by a single point in all possible ways. Here \(\neg^0\) and \(\neg^1\) mean no symbol and \(\neg\), respectively.

\begin{definition}
Let \(A_\sf{SI}\) be the set of first-order sentences of the form
\[
\forall x_1,\dots, x_m\colon \bigcurlywedge_{i=1}^m x_i \not\approx e  \to \exists y \colon y \not\approx e \curlywedge \Bigg( \bigcurlywedge_{i=1}^m y \not\approx x_i \Bigg)\curlywedge \neg^c T(y,y,y) \curlywedge{} 
\]
\[ \Bigg( \bigcurlywedge_{i = 1}^{m} \neg^{c_i} T(x_i,y,y) \Bigg) \curlywedge \Bigg( \bigcurlywedge_{1 \le i \le j \le m} \neg^{c_{ij}} T(x_i, x_j, y) \Bigg),
\]
where \(m \in \omega\) and \(c, c_i, c_{ij} \in \{0,1\}\), for all \(1 \le i \le j \le m\).
\end{definition}

\begin{lemma}\label{axfraisse}
Let \(\bf{L}\) be countable model of \textup{(II)}, \textup{(IP)}, and \(f(x) \approx x\). Then \(\bf{L} \cong \bf{L}_\sf{SI}\) if and only if \(\bf{L} \models A_\sf{SI}\).
\end{lemma}

\begin{proof}
For the forward direction, assume that \(\bf{L} \cong \bf{L}_\sf{SI}\). Then \(\bf{L}\) is a Fra{\"i}ss{\'e} limit of \(\sf{FSIAS}_e\), hence the age of \(\bf{L}\) is \(\sf{FSIAS}_e\) and \(\bf{L}\) is ultrahomogeneous.  Let \(n \in \omega\), let \(c,c_i, c_{ij} \in \{0,1\}\), for all \(1 \le i \le j \le n\), and let \(u_1,\dots, u_n \in L\setminus \{e^\bf{L}\}\). Let \(\bf{U}\) denote the substructure of \(\bf{L}\) generated by \(U \deq \{u_1,\dots,u_n\}\). Fix some element \(v \notin U\) and define \(\bf{V} \deq \inn{V; f^\bf{V}, e^\bf{V}, T^\bf{V}}\), where \(V \deq U \cup \{e^\bf{F},v\}\), \(f^\bf{V} = \id_V\), \(e^\bf{V} = e^\bf{L}\), and \(T^\bf{V}\) is given by
\[
 \begin{cases} T^\bf{U} \cup [v,e^\bf{L},v]\cup [v,v,v]\cup \Big(\bigcup \{ [u_i,v,v] \mid c_i = 0\} \Big)\cup \Big(\bigcup\{ [u_i, u_j, v] \mid c_{ij} = 0\}\Big) & \rm{if } c = 0 \\
 T^\bf{U} \cup [v,e^\bf{L},v] \cup \Big(\bigcup \{ [u_i,v,v] \mid c_i = 0\} \Big)\cup \Big(\bigcup\{ [u_i, u_j, v] \mid c_{ij} = 0\}\Big) & \rm{if } c = 1.
\end{cases}
\]
Since the age of \(\bf{L}\) is \(\sf{FSIAS}_e\), we must have \(\bf{U} \in \sf{FSIAS}_e\), so it is easy to see that \(\bf{V} \in \sf{FSIAS}_e\). Again, since the age \(\bf{L}\) is \(\sf{FSIAS}_e\), there is a substructure \(\bf{W}\) of \(\bf{L}\) that is isomorphic~to~\(\bf{V}\). Let~\(\mu \colon \bf{V} \to \bf{W}\) be an isomorphism. Then \(\mu \circ\imath_U \) is an isomorphism from \(\bf{U}\) to the substructure of \(\bf{L}\) generated by \(\mu[U]\). As \(\bf{L}\) is ultrahomogeneous, \(\mu \circ\imath_U \) extends to an automorphism, say \(\nu\). Then, by construction, \(\nu^{-1}(\mu(v))\) is the witness to the sentence from \(A_\sf{SI}\) given by \(n\), \(c\), and each \(c_i\), and \(c_{ij}\), when choosing  \(x_i = u_i\), for each \(1 \le i \le n\). Thus, \(\bf{L} \models A_\sf{SI}\).

Conversely, assume that \(\bf{L} \models A_\sf{SI}\). To show that \(\bf{L} \cong \bf{L}_\sf{SI}\), we need to show that \(\sf{FSIAS}_e\) is the age of \(\bf{L}\) and that \(\bf{L}\) is ultrahomogeneous. As \(\bf{L}\) is a symmetric model of (IP) and (II), the age of \(\bf{L}\) is a subset of \(\sf{FSIAS}_e\). Assume, for a contradiction, that this inclusion is proper. Let \(\bf{U}\) be an element of minimal size in \(\sf{FSIAS}\) that is not in the age of \(\bf{L}\). Clearly, \(|U| > 1\). Now, let \(u \in U \setminus \{e^\bf{U}\}\) and let \(\bf{V}\) denote the substructure of \(\bf{U}\) generated by \( V \deq U \setminus \{u\}\). By our minimality assumption, \(\bf{V}\) embeds into \(\bf{L}\). Let \(\mu \colon \bf{V} \to \bf{L}\) be such an embedding. Let \(m \deq |V|-1\), let \(\{v_1,\dots,v_m\}\) be an enumeration of \(V\setminus \{e^\bf{V}\}\), let
\[
c \deq \begin{cases} 0 & \rm{if } [u,u,u] \subseteq T^\bf{U}\\ 1 & \rm{if } [u,u,u] \nsubseteq T^\bf{U}, \end{cases}
\]
let
\[
c_i \deq \begin{cases} 0 & \rm{if } [v_i,u,u] \subseteq T^\bf{U} \\ 1 & \rm{if } [v_i,u,u] \nsubseteq T^\bf{U}, \end{cases}
\]
for all \(1 \le i \le m\), and let
\[
c_{ij} \deq \begin{cases} 0 & \rm{if } [v_i,v_j,u] \subseteq T^\bf{U} \\ 1 & \rm{if } [v_i,v_j,u] \nsubseteq T^\bf{U}, \end{cases}
\]
for all \(1 \le i \le j \le m\). As \(\bf{L} \models A_\sf{SI}\),  \(\bf{L}\) satisfies the sentence defined by \(c\) and each \(c_i\) and \(c_{ij}\), so there is a witness, say \(y\), for the choice of \(x_i \deq \mu(v_i)\), for all \(1 \le i \le m\). By construction, the substructure of \(\bf{L}\) generated by \(\mu[V] \cup \{y\}\) is isomorphic to \(\bf{U}\). Thus, \(\bf{U}\) embeds into \(\bf{L}\), so \(\bf{U}\) is in the age of \(\bf{L}\), contradicting our assumption.

Lastly, based on Lemma \ref{homogequiv}, it will be enough to establish that \(\bf{L}\) is weakly homogeneous.  Let~\(\bf{U} \le \bf{V}\) be finitely generated substructures of \(\bf{L}\) and let \(\mu \colon \bf{U} \to \bf{L}\) be an embedding. Note that since \(\bf{L}\) is symmetric, \(e\) is the only new element that can be generated by a subset.  If \(\bf{U} = \bf{V}\), then we are done. Assume that \(\bf{U} \neq \bf{V}\) and fix some \(v \in V \setminus U\). Let \(m \deq |U|-1\), let \(\{u_1,\dots,u_m\}\) be an enumeration of \(U\setminus \{e^\bf{U}\}\), let

\[
c \deq \begin{cases} 0 & \rm{if } [v,v,v] \subseteq T^\bf{U}\\ 1 & \rm{if } [v,v,v] \nsubseteq T^\bf{U}, \end{cases}
\]
let
\[
c_i \deq \begin{cases} 0 & \rm{if } [u_i,v,v] \subseteq T^\bf{U} \\ 1 & \rm{if } [u_i,v,v] \nsubseteq T^\bf{U}, \end{cases}
\]
for all \(1 \le i \le m\), and let
\[
c_{ij} \deq \begin{cases} 0 & \rm{if } [u_i,u_j,v] \subseteq T^\bf{U} \\ 1 & \rm{if } [u_i,u_j,v] \nsubseteq T^\bf{U}, \end{cases}
\]
for all \(1 \le i \le j \le m\).  As \(\bf{L} \models A_\sf{SI}\),  \(\bf{L}\) satisfies the sentence defined by \(c\) and each \(c_i\) and \(c_{ij}\), so there is a witness, say \(y\), for the choice of \(x_i = \mu(u_i)\), for all \(1 \le i \le m\). By construction, the map \(\nu \colon U \cup \{v\} \to L\) given by
\[
\nu(x) = \begin{cases} \mu(x) & \rm{if } x \in U \\ y & \rm{if } x = v \end{cases}
\]
embeds the substructure of \(\bf{V}\) generated by \(U \cup \{v\}\) into \(\bf{L}\). By assumption, \(\bf{V}\) is finite, hence \(\mu\) can be extended to an embedding \(\nu \colon \bf{V} \to \bf{L}\) by repeating this construction. Thus, \(\bf{L}\) is weakly homogeneous, which is what we wanted to show.
\end{proof}

Based on Theorem \ref{fraisse}, we have the following.

\begin{corollary}
Together, \textup{(IP)}, \textup{(II)}, \(f(x) \approx x\), and \(A_\sf{SI}\) form a \(\aleph_0\)-categorical and therefore complete theory.
\end{corollary}

\begin{lemma}
\(A_\sf{SI} \subseteq S_\sf{SI}\).
\end{lemma}

\begin{proof}
Let \(n \in \bb{N}\),  let \(m \in \omega\), let \(c, c_i, c_{ij} \in \{0,1\}\), for all \(1 \le i \le j \le m\), and let \(\sigma\) be the sentence from \(A_\sf{SI}\) defined by these parameters. Clearly, given non-identity elements \(x_1,\dots,x_m, y \in \{1,\dots,n\}\) such that \(y \neq x_i\), for all \(1 \le i \le n\), at most 
\[
1+m+\frac{m^2+m}{2} = \frac{m^2+3m+2}{2}
\]
cycles are forced to be consistent for \(\sigma\) to be satisfied for these given choices. Thus, the fraction of structures failing the sentence with these choices is below \(1-2^{-(m^2+3m+2)/2}\). There are \((n-1)^m\) ways to select \(x_1,\dots,x_m\), \(n-m-1\) ways to select \(y\) given \(x_1,\dots,x_m\), and the choices of \(y\) are independent once \(x_1,\dots,x_n\) are selected. Based on these results, the fraction of structures not modelling \(\sigma\) is bounded above by
\[
(n-1)^m (1-2^{-(m^2+3m+2)/2})^{n-m-1}.
\]
Clearly, \((n-1)^m \le n^{m} = 2^{m\log_2(n)}\) and \(n-m -1 < n\), so this quantity is below
\[
2^{m \log_2(n) + n\log_2(1-2^{-(m^2+3m+2)/2}) }.
\]
Since \(1-2^{-(m^2+3m+2)/1} < 1\), we have \(\log_2(1 -2^{-(m^2+3m+2)/2}) < 0\), hence
\[
\lim_{n \to \infty} 2^{m \log_2(n) + n\log_2(1-2^{-(m^2+3m+2)/2}) } = 0.
\]
By the Squeeze Principle, the fraction of structures not modelling \(\sigma\) tends to 0. Thus, \(A_\sf{SI} \subseteq S_\sf{SI}\), as claimed.
\end{proof}

So, based on Proposition \ref{fraisse} and Lemma \ref{axfraisse}, we have the following.

\begin{corollary}\label{01e}
\(S_\sf{SI}\) is \(\aleph_0\)-categorical, and therefore complete. Thus, \(\sf{FSIAS}_e\) has a \(0\)--\(1\) law.
\end{corollary}

This result can be translated to a result for \(\sf{FSIAS}\).

\begin{corollary}
\(\sf{FSIAS}\) has a \(0\)--\(1\) law.
\end{corollary}

\begin{proof}
Let \(\bf{U} \in \sf{FSIAS}\),  let \(e^\bf{U}\) be the unique element of \(I\), let \(\bf{U}_e \deq \inn{U; f^\bf{U}, e^\bf{U}, T^\bf{U}}\), let \(\sigma\) be a \(\{f,T,I\}\)-sentence, and let \(\sigma_e\) be the \(\{f,e,T\}\)-sentence obtained from \(\varphi\) by replacing all occurences of \(I(x)\), for some  \(x\), with \(x \approx e\). By construction,  \(\bf{U} \models \varphi\) if and only if \(\bf{U}_e \models \varphi_e\). As there is a one-to-one correspondence between isomorphism classes in \(\sf{FSIAS}\) and \(\sf{FSIAS}_e\), this observation and Corollary~\ref{01e} tell us that \(\sf{FSIAS}\) has a \(0\)--\(1\) law, as required.
\end{proof}

Hence, by Theorem \ref{aasi}, we have the following.

\begin{corollary}\label{01}
\(\sf{FAS}\) has a \(0\)--\(1\) law.
\end{corollary}

\section{Further work}

Perhaps the most obvious open problem in this area is problem of determining whether almost all nonassociative relation algebras are representable; this problem is mentioned by Maddux in \cite{perspecRA} and by Hirsch and Hodkinson in \cite{rabg}. A possible first step to solving this problem could be solving the corresponding problems for the classes of feebly and qualitatively representable algebras introduced by Hirsch, Jackson, and Kowalski in \cite{quali}.

\begin{problem}
Determine whether almost all nonassociative relation algebras are feebly, qualitatively, or (strongly) representable.
\end{problem}
 
Determining whether or not Corollary~\ref{01} extends to the classes of nonassociative relation algebras and relation algebras would also be an interesting problem.

\begin{problem}
Determine whether \(\sf{NA}\) and \(\sf{RA}\) have \(0\)--\(1\) laws.
\end{problem}

\section{Acknowledgements}

This is a pre-print of a paper contributed to the Journal of Symbolic Logic. The final authenticated version is available at: \url{https://doi.org/10.1017/jsl.2021.81}. The author would like to thank Roger Maddux for suggesting the result from Corollary~\ref{symint}.


\begin{thebibliography}{99}

\bibitem{bad}
      Badesa, C.:
      The birth of model theory. Löwenheim's theorem in the frame of the theory of relatives. 
      Translated from the Spanish by Michael Maudsley. 
      Princeton University Press, Princeton (2004)

\bibitem{bellburr}
     Bell, J.P., Burris, S.N.:
     Compton's method for proving logical limit laws.
     Model theoretic methods in finite combinatorics.
     Contemp. Math.
     \textbf{558}, Amer. Math. Soc., Providence (2011)

\bibitem{carnap}
     Carnap, R.:
     Logical foundations of probability.
     University of Chicago Press, Chicago (1950)

\bibitem{comp1}
     Compton, K.:
     A logical approach to asymptotic combinatorics. I. First-order properties. 
     Adv. in Math.
     \textbf{65}, 65--96 (1987)

\bibitem{comp2}
     Compton, K.:
     A logical approach to asymptotic combinatorics. II. Monadic second-order properties.
     J. Combin. Theory Ser. A
     \textbf{50}, 110--131 (1989)

\bibitem{morg}
      De Morgan, A.:
      Syllabus of a proposed system of logic.
      Walton and Maberly, London (1860)

\bibitem{mtprob}
     Fagin, R.: 
     Probabilities on finite models.
     J. Symb. Logic
     \textbf{41}, 50--58 (1976)

\bibitem{mtnumb}
     Fagin, R.:
     The number of finite relational structures.
     Discrete Math.
     \textbf{19}, 17--21 (1977)

\bibitem{fraisse}
      Fra{\"i}ss{\'e}, R.:
      Sur l'extension aux relations de quelques propri{\'e}t{\'e}s des ordres (French).
      Ann. Sci. Ecole Norm. Sup. (3)
     \textbf{71}, 363--388 (1954)

\bibitem{2kinds}
     Freese, R.:
     On two kinds of probability in algebra.
     Algebra Universalis
     \textbf{27}, 70--79 (1990)

\bibitem{intRA}
     Givant, S.:
     Introduction to relation algebras.
     Springer, Cham (2017)    

\bibitem{og1}
     Glebski\u{\i}, Y.V., Kogan, D.I., Liogon'ki\u{\i}, M.I., Talanov, V.A.:
     Range and degree of realizability of formuas in the restricted predicate calculus (Russian).
     Kibernetika (Kiev),
     no. 2, 17--27 (1969)

     English translation: Cybernetics \textbf{5}, 142--154 (1969)

\bibitem{rabg}
     Hirsch, R., Hodkinson, I.:
     Relation algebras by games.
     Studies in Logic and the Foundations of Mathematics \textbf{147}.
     North-Holland, Amsterdam (2002)

\bibitem{quali}
      Hirsch, R., Jackson, M.G., Kowalski, T.:
      Algebraic foundations for qualitative calculi and networks.
      Theoret. Comput. Sci.
     \textbf{768}, 99--116 (2019)

\bibitem{lsmt}
     Hodges, W.:
     A shorter model theory.
     Cambridge University Press, Cambridge (1997)  

\bibitem{og2}
     Liogon'ki\u{\i}, M.I.:
     On the question of quantitative characteristics of logical formulas (Russian).
     Kibernetika (Kiev),
     no. 3, 16-22 (1970) 

     English translation: Cybernetics \textbf{6}, 205--211 (1970)

\bibitem{lyndon}    
     Lyndon, R.C.:
     The representation of relation algebras.
     Ann. Math. \textbf{51}, 707--727 (1950)
    
\bibitem{perspecRA}
     Maddux, R.D.:
     A perspective on the theory of relation algebras.
     Algebra Universalis
     \textbf{31}, 456--465 (1994)        

\bibitem{fira}
     Maddux, R.D.:
     Finite integral relation algebras.
     In: Charleston, S.C. (ed.) Universal Algebra and Lattice Theory, 175--197. 
     Springer, Berlin (1985)

 \bibitem{madduxRA}
     Maddux, R.D.:
     Relation algebras.
     Studies in Logic and the Foundations of Mathematics \textbf{150}.
     North-Holland, Amsterdam (2006)        
    
\bibitem{varcontRA}
     Maddux, R.D.:
     Some varieties containing relation algebras.
     Trans. Amer. Math. Soc.
     \textbf{272}, 501--526 (1982)    

\bibitem{peirce}
     Peirce, C.S.:
     Description of a notation for the logic of relatives, resulting from an amplification of the conceptions of Boole's calculus of logic.
     Mem. Amer. Acad. Arts n.s.
     \textbf{9}, No. 2, 317--378 (1873)

\bibitem{russell}
     Russell, B.:
     The principles of mathematics.
     Cambridge University Press, Cambridge (1903)
     
\bibitem{schroder}     
     Schr{\"o}der, E.:
     Vorlesungen und Logik der Relative, der Vorlesungen {\"uber} die Algerbra der Logik 3.
     
\bibitem{og}
     Tarski, A.: 
     On the calculus of relations.
     J. Symb. Logic
     \textbf{6}, 73--89 (1941)

\end{thebibliography}
\end{document}